\documentclass[11pt]{article}
\usepackage{color}
\usepackage{relsize}
\usepackage{amsmath,amsthm,amssymb}
\usepackage{amsfonts}
\usepackage{enumerate}
\usepackage{appendix}

\usepackage{graphicx}

\newtheorem{thm}{Theorem}[section]

\newtheorem{theorem}[thm]{Theorem}
\newtheorem{corollary}[thm]{Corollary}
\newtheorem{lemma}[thm]{Lemma}

\newtheorem{example}[thm]{Example}

\theoremstyle{definition}

\newtheorem{definition}[thm]{Definition}

\numberwithin{equation}{section}
\def\bR{\mathbb{R}}

\def\om{\Omega}

\def\p{\partial}
\def\P{\mathbb{P}}
\def\R{\mathbb{R}}
\def\bR{\mathbb{R}}

\theoremstyle{remark}
\newtheorem{rem}{Remark}[section]
\newtheorem{remark}[rem]{Remark}

\usepackage[colorlinks, linkcolor=black, citecolor=black]{hyperref}

\begin{document}
	
	\title{Backward Uniqueness for 3D Navier-Stokes Equations with Non-trivial Final Data and Applications}
	
	\author{Zhen Lei \footnotemark[1]\ \footnotemark[2]
		\and Zhaojie Yang    \footnotemark[1]\ \footnotemark[3]
		\and Cheng Yuan  \footnotemark[1]\ \footnotemark[4]
	}
	\renewcommand{\thefootnote}{\fnsymbol{footnote}}
	\footnotetext[1]{School of Mathematical Sciences; LMNS and Shanghai Key Laboratory for Contemporary Applied Mathematics, Fudan University, Shanghai 200433, P. R.China.} \footnotetext[2]{Email: zlei@fudan.edu.cn}
	\footnotetext[3]{Email: yangzj20@fudan.edu.cn}
	\footnotetext[4]{Email: cyuan22@m.fudan.edu.cn}
	
	\date{\today}
	
	\maketitle
	
	\begin{abstract}
		Presented is a backward uniqueness result of bounded mild solutions of 3D Navier-Stokes Equations in the whole space with non-trivial final data. A direct  consequence is that a solution must be axi-symmetric in $[0, T]$ if it is so at time $T$.  The proof is based on a new weighted estimate which enables to treat terms involving Calderon-Zygmund operators. The new weighted estimate is expected to have certain applications in control theory when classical Carleman-type inequality is not applicable.
	\end{abstract}

	\section{Introduction}\label{intro}
	
	In this paper, we consider the backward uniqueness problem for the following 3D Navier-Stokes equations in $\R^3\times [0,T]$:
	\begin{equation}\label{ns}
		\begin{cases}
			&\partial_{t}u - \Delta u + (u\cdot \nabla)u + \nabla p = f,\\
			&\nabla\cdot u = 0,\\
			&u|_{t=T} = g(x),
		\end{cases}			
	\end{equation}
	where $u$ denotes the velocity field, $p$ is the pressure, and $f$ is the forcing term. Here, we are interested in bounded mild solutions. Hence, it is assumed that the final data $g$ is a bounded function, and the existence of bounded mild solution $u$ is a known fact. More precisely, one has
	\begin{equation*}
		p = \mathcal{K}:(u\otimes u - F),
	\end{equation*}
	and
	\begin{equation}\nonumber
		u(t) = e^{(t-s)\Delta} u(s) +  \int_{s}^{t} e^{(t-\tau)\Delta} \mathbb{P}\big[\nabla\cdot(u\otimes u)(\cdot, \tau) + f(\cdot, \tau)\big]d\tau,\quad 0 \leq s \leq  t \leq T,
	\end{equation}
	with $f = \nabla\cdot F$, where $\mathcal{K} = (-\Delta)^{-1}\nabla^2$ is the standard Riesz operator and $\mathbb{P}$ is the Leray projector. For more details, see Section 2.

	The backward uniqueness (BU) problem asks whether there are two different bounded mild solutions $u_{1}$ and $u_{2}$ to \eqref{ns} with the same final data $g$ at some time $T\in(0,+\infty)$. In this article we are interested in the case of non-trivial final data $g$. We say that the (BU) property holds if $u_1$ and $u_2$ must coincide.
	
	\subsection{Background}

	The (BU) property is closely connected to
	internal control theory.  In fact, if $u_1, u_2$ are two different bounded mild solutions to \eqref{ns}, denote  $u:=u_2-u_1$. It is clear that $u$ satisfies the following modified Navier-Stokes equation with internal control:
	\begin{equation}\label{ns1}
		\begin{cases}
			&\p_{t}u + u\cdot\nabla u- \Delta u + \nabla \pi = -u_1\cdot \nabla u-u\cdot\nabla u_1,\\
			&\nabla\cdot
			u=0,\\
			&u|_{t=t_{0}}=u_{0}.
		\end{cases}
	\end{equation}
	Indeed, by forgetting about the relation between $u$ and $u_1$, one can think of $u_1$  as an internal control function.
	That the (BU) property holds means  that under suitable constraints on the control functions $u_1$, the  \emph{null controllability} for \eqref{ns1} is impossible, that is, for any nontrivial data $u(t_{0})$, \eqref{ns1} can not achieve the trivial state $u(\cdot, T)\equiv 0$ for any $T > t_0$. 
	
	Besides, the study of backward uniqueness has its own significance in the PDE's field, especially in the regularity theory of parabolic equations. For instance, the regularity theory of Navier-Stokes equations in critical spaces (\cite{2002On},  \cite{Carlos2011An}, \cite{2019Quantitative}, \cite{Palasek_2021}, \cite{Palasek_2022}), regularity of harmonic maps \cite{2008W}, and blow up for semi-linear heat equation \cite{2011MN}. 
	
	In the literature, the backward uniqueness \& unique continuation properties of linear parabolic equations have been widely considered.  Lin \cite{1990A} proved the strong unique continuation for linear parabolic equations of the form $\p_{t}u-\Delta u+V(x)u=0$, where  time-independent coefficients $V(x)\in L_{x}^{(n+1)/2}$. See also \cite{Escauriaza2000Carleman}, \cite{0Carleman} for results on time-dependent potentials. For equations of the form $\p_{t}u-\Delta u+a(t,x)u+b(t,x)\cdot\nabla u=0$, Poon \cite{1996P} proved that if $a, b$ are bounded, the bounded solution $u$ in $\mathbb{R}^n \times \mathbb{R}^+ $ is identically zero if it vanishes of infinite order at any point.
	See also \cite{1998A} for improved results allowing a suitable growth condition on $u$ and \cite{1997C1},  \cite{2003Unique} for analogous results on divergence-form backward operator $P=\p_{t}+\p_{i}(a^{ij}(x,t)\p_{j})$.

	In the remarkable work \cite{2002On}, Escauriaza, Seregin and Sverák proved the global regularity of $L^\infty L^3$ solutions of \eqref{ns} with null force, where backward uniqueness arguments (proved in \cite{2004Backward}) played a crucial role. Indeed, they proved that if a function $u$ satisfies $|\Delta u+\p_{t}u|\lesssim (|u|+|\nabla u|)$ and a certain growth condition in $\R^{n}\setminus B_{R}\times [0,T]:=Q_{R,T}$, then $u(x,0)\equiv 0$ in $\R^{n}\setminus B_{R}$ implies $u(x,t)\equiv 0$ in $Q_{R,T}$. The result can be applied to prove the vanishing of the vorticity $\omega\triangleq\nabla \times u$ in $Q_{R,T}$ for large $R$ provided $\omega$ vanishes at $t=T$. Their proof relies on some Carleman-type inequalities, which have also been widely used in control theory. See also \cite{2012LSV} for the use of Carleman estimates proving backward uniqueness of parabolic equations in a cone and \cite{2019Quantitative},  \cite{Palasek_2021}, \cite{Palasek_2022} for proving quantitative regularity and blowup criterion for Navier-Stokes equations.
	
	
	In addition to the important role of (BU) property played in the regularity theory of Navier-Stokes Equations, we point out that the (BU) property with \emph{non-trivial final data} studied in this article has its own obvious significance. For instance, one may ask the following question: 
	
	
	\begin{center}{(Q)}
		Suppose that $f = 0$ and $u(t, x)$ is a bounded mild solution to the Navier-Stokes equations with $u(T, x)$ being an axi-symmetric vector   without swirl. Must $u(t, x)$ be axi-symmetric and swirl-free?
	\end{center}
	To answer this question, one has to first answer whether $u$ is axi-symmetric and then apply standard backward uniqueness result after various surgeries. The proof will be presented after stating the main Theorem \ref{main}. Let us emphasis that a crucial difference from the trivial final data case is that one has to treat terms involving Calderon-Zygmund operators. Since the classical Carleman type inequality involves the exponential weight which does not belong to $\mathcal{A}_{p}$ class for any $p \in [1, \infty]$, we need find new weight to overcome this essential difficulty. The new weighted estimate is expected to have certain applications in control theory when classical Carleman-type inequality is not applicable.
	\subsection{Main Result}
	In this paper, we will prove the backward uniqueness of forced Navier-Stokes equations \eqref{ns} for bounded mild solutions, with nontrivial final data.  Our main result is the following theorem:
	\begin{thm}\label{main}
		Assume $u_1(x,t), u_2(x,t)$ are two bounded mild solutions of (\ref{ns}), such that $\nabla \times u_1$ and $\nabla \times u_2$ are also bounded. If $u_1(x,T)= u_2(x,T)$ for some $T>0$, then $(u_1, \nabla p_1)$ must coincide with $(u_2, \nabla p_2)$ at any time $0\le t\le T$.	
	\end{thm}
	\begin{remark}Our result only requires the vorticity to be bounded, this assumption is weaker than several classical results, such as in \cite{1996P}.  It is clear that  the boundedness condition on vorticity can be removed if $f$ vanishes or has certain regularity. For instance, if  $f=\nabla\cdot F$ with $F\in L^{\infty}_{t}W^{1, p}_x$ for $3< p\leq \infty $, then one can easily derive that $u_1, u_2\in \mathbf{C}\big((0, T], W^{1,\infty}(\R^3)\big)$ (see \cite{2009SS}).
	\end{remark}
	Note that the \emph{mild} assumption on the solutions in  Theorem \ref{main} is necessary. In view of the following simple counter-example, one can not expect that the (BU) properties holds true for general $L^\infty$ solutions:
	\begin{example}\label{counter}
		Let  $u(x,t)=(h(t),0,0)$, with $h(T)=0, h'(T)=0$, and $ h(t)\not\equiv 0$ for $t\neq T$. Let $p(t,x)=-h'(t)x_{1}$, then $(u,p)$ is a classical solution of (\ref{ns}) vanishing at $t=T$, and is obviously nontrivial. 
	\end{example}
	
	
	The backward uniqueness holds for a large class of solutions of \eqref{ns} such that the Ladyzhenskaya–Prodi–Serrin regularity criteria holds. Indeed, the boundedness of $u$  is a (non-trivial) consequence of these conditions stated below. For instance, see  \cite{1982CKN}, \cite{2002On}, or \cite{2016RG} for proofs.
	\begin{corollary}\label{coro}
		Assume $u(x,t),v(x,t)$ are mild solutions of \eqref{ns} with $f = 0$ such that $u(x,t),v(x,t)\in L^{p}\big([0,T];L^{q}(\R^{3})\big)$ for $2/p+3/q=1, 3 \leq q \leq \infty$ Then, $u(x,t)$ must coincide with $v(x,t)$ at any time $0\le t\le T$,  provided that $u(x,T)\equiv v(x,T)$ for some $T>0$. 
	\end{corollary}
	
	\begin{rem}
		The regularity assumptions on $u, v$ and $f$ in Corollary \ref{coro} can be replaced by various functional frameworks. See \cite{2016RG} for a comprehensive survey on the regularity criteria of \eqref{ns}.
	\end{rem}
	
	Now let us first assume the validity of Theorem \ref{main} and present the answer of $(Q)$.
	\begin{proof}
		First of all, we show that $u(t, x)$ is also axi-symmetric for all $t \in [0, T]$. Recall that a vector $u$ is called axi-symmetric if
		$$u(x) = u^re_r + u^\theta e_\theta + u^ze_z,$$
		where $u^r, u^\theta,  u^z$ are all functions of $(r, z)$ but not depend on $\theta$, with $r = \sqrt{x_1^2 + x_2^2}$, $\theta = \arctan\frac{x_2}{x_1}$, and $z = x_3$.
		Then any rotation by $\phi$ degrees around z-axis can be described by the matrix
		$$
		S_{\phi}=\begin{pmatrix}
			\cos \phi & \sin \phi & 0 \\
			-\sin \phi & \cos \phi & 0 \\
			0 & 0 & 1
		\end{pmatrix}.
		$$
		It is easy to check that the Navier-Stokes equation is invariant under axi-symmetric transformation. Hence if $u(x, t)$ is a solution of the N-S equation, then so does $ u_\phi := S_\phi u(S_{-\phi} x, t)$ . Since $u(x, T)$ is an axi-symmetric vector, it holds that
		$$
		S_\phi u(S_{-\phi} x, T) = u(x, T),
		$$
		or equivalently $u_\phi(x, T) = u(x, T)$ for any $\phi \in [0, 2\pi]$. By Theorem \ref{main}, we have $u_\phi(x, t) = u(x, t)$ for any $t < T$, which implies that $u(x, t)$ is axi-symmetric for all $\phi \in [0, 2\pi]$ and $t<T$.

		Next, denote $G = \frac{u^\theta}{r}$. It is clear that $G$ is governed by (\cite{2008HL, 2008HLL})
		$$G_t - \Delta G - \frac{2}{r}\partial_rG + b\cdot\nabla G + \frac{2u^r}{r}G = 0,$$
		with $b=u^r e_r + u^z e_z. $ Since here we are dealing with bounded mild axi-symmetric solutions, it is clear that $b$ and $\frac{u^r}{r}$ are bounded.
		Similarly as in \cite{2008HLL, 2015L}, one may look on the above equation as a five-dimensional one and derive that
		$$|G_t - \Delta_5G| \leq C(|G| + |\nabla G|).$$
		Here 
		$$\Delta_5 = \Delta + \frac{2}{r}\partial_r,\quad \Delta = \partial_r^2 + \frac{1}{r}\partial_r + \partial_z^2.$$
		Then one may apply classical backward uniqueness result (\cite{1996P}) to derive that $G \equiv 0$.
		
		Hence, the answer of (Q) is yes.
	\end{proof}

	\subsection{Main Ingredients}
	Now let us explain the main difficulties of this paper and the strategies used to overcome them. \\
	\textit{(a) Invalidity of Carleman estimates.} We emphasize that those above-mentioned literature do not need to treat terms involving the Calderon-Zygmund operator and hence Carleman-type inequality can be applied. Since  exponential-type weights appeared in Carleman type estimates do not belong to any $\mathcal{A}_p$-class for  $1 < p < \infty$, one can not apply Carleman type estimates to treat the terms involving Calderon-Zygmund operator. Unfortunately, for \eqref{ns} when the final data is nontrivial, the Calderon-Zygmund operator will inevitably comes in.  
	
	To overcome it, we establish a new weighted estimate (see Theorem \ref{weight}).  Compared with the Carleman estimate in \cite{2004Backward}, we use the  weights with polynomial spatial decay. This idea behind is the following: the Carleman inequality in \cite{2004Backward} is applied to the heat equation whose kernel has exponential decay. Hence, weights with exponential decay appear naturally in \cite{2004Backward} in view of the heat kernel. In the current setting, due to the presence of the Stokes operator, whose corresponding semi-group only enjoys an algebraic decay rate, it is natural to choose a polynomial type weight $\langle x\rangle^{-k}$.\\
	\textit{(b) Invalidity of $\mathcal{A}_{p}$ weights.} Since the solutions we dealt with here are bounded mild solutions and may not have any decay when $x$ tends to infinity, the spatial weight $\langle x\rangle^{-k}$ in our proof must be integrable over the whole space, which implies $k>3$. On the other hand, due to the appearance of the pressure term, which produces non-local effects, we have to control terms like
	\begin{equation}\label{intro-weight}
		\int_{\R^{3}}\langle x\rangle^{-k}|\mathcal{R}f|^{2}dx\lesssim \int_{\R^{3}}\langle x\rangle ^{-k}|f|^{2}dx,
	\end{equation}
	where $\mathcal{R}$ is a Calderon-Zygmund operator. Unfortunately, when $k>3$, $\langle x\rangle^{-k}$ does not belong to any $\mathcal{A}_p$ class (see \cite{2014GL}), which causes extra technical difficulties in estimating \eqref{intro-weight} . To overcome it, we exploit the inherent divergence structure  in  nonlinear terms and use elementary integration by parts techniques, see Lemma \ref{lem-weight}.
	

	\subsection{Structure of the Paper}
	The rest of the paper is organized as follows. In section \ref{pre}, we will define the $L^\infty$ mild solution of \eqref{ns}, and describe some basic properties.  In section \ref{carleman}, we will establish a weighted-type estimate, which is the key tool in our proof. In section \ref{pf}, we will give a rigorous proof of the main theorem \ref{main}. \\
	\textit{Notations.} In this paper, we denote $A\lesssim B$ to mean $A\le c B$ for some absolute constant $c$. We denote $A\approx B$ to mean that $A\lesssim B$ and $B\lesssim A$.
	
	\section{Preliminary}\label{pre}
	In this section, we will give a brief definition of $L^\infty$ mild solutions of \eqref{ns}.
	\begin{definition}$u(t)$ is defined to be an $L^\infty$ mild solution of the Cauchy problem (\ref{ns}), if $u\in L^{\infty}_{t,x}$ satisfies \eqref{ns} in the distributional sense and the following integral equation
		\begin{equation}\label{mild}
				u(t) = e^{(t-s)\Delta} u(s) +  \int_{s}^{t} e^{(t-\tau)\Delta} \mathbb{P}\big[\nabla\cdot(u\otimes u)(\cdot, \tau) + f(\cdot, \tau)\big]d\tau,\quad 0 \leq s \leq  t \leq T.
		\end{equation}
	\end{definition}
	\begin{rem}
		Here $e^{t\Delta}$ is the standard heat semigroup, and $\mathbb{P}$ denotes the Leray projection on $\R^{3}$. Cleary, $\P$ is a Calderon-Zugmund type singular integral operator and can be formally written by
		\begin{equation*}
			\mathbb{P}=I+\nabla(-\Delta)^{-1}(\nabla \cdot).
		\end{equation*}
	\end{rem}
	
	\begin{remark} From the definition of mild solution (\ref{mild}), the pressure term can be recovered by
		\begin{equation*}
			\nabla p = \nabla(-\Delta)^{-1}\p_{i}\p_{j}u^{i}u^{j} - \nabla(-\Delta)^{-1}(\nabla \cdot f),
		\end{equation*}
		therefore we are able to estimate the pressure in terms of the velocity filed.
		
	\end{remark}	
	\begin{rem}	
		It is clear that (for instance, see \cite{2009SS}) \eqref{mild} can be written in a convolution form. Explicitly,
		\begin{align}\label{kernel}
			\int_{0}^{t} e^{(t-s)\Delta}\mathbb{P}\big[(u\cdot \nabla)u(\cdot,s)\big]ds&=\int_{0}^{t}\int_{\R^{3}}\Big(\nabla O(x-y,t-s)::u\otimes u(y,s)\Big)dyds,\\
			\int_{0}^{t}e^{(t-s)\Delta}\mathbb{P}(f)(\cdot,s)ds&=\int_{0}^{t}\int_{\R^{3}}O(x-y,t-s)\cdot f(y,s)dyds,
		\end{align}
		where the kernel satisfies following estimates:
		\begin{equation}\label{est}
			|\nabla^{k}O(x,t)|\lesssim (|x|+\sqrt{t})^{-3-|k|},
		\end{equation}
		which implies \eqref{mild} is well defined for $u\in L^{\infty}_{t,x}.$ 
	\end{rem}

	\section{Weighted $L^{2}$ Estimates}\label{carleman}
	In this section, we will prove the following weighted inequality.
	\begin{theorem}\label{weight}
		Assume $u(x,t)\in C_{c}^{\infty}\big(\R^{n}\times (T_{-},T^{+})\big)$, where $[T_{-},T^{+}]\subset (0,\infty)$, $T^{+}$ is sufficiently small. Then for any large number $a>0$ and any $k\ge 0$, the following inequality holds:
		\begin{align}
			&\int_{\bR^{n}\times(T_{-},T^{+})}h(t)^{-(2a+1)}(1+|x|^{2})^{-k}\Big((a+1)u^{2}+|\nabla u|^{2}\Big)dxdt\nonumber\\
			&\quad \quad\lesssim \int_{\bR^{n}\times(T_{-},T^{+})}h(t)^{-2a}(1+|x|^{2})^{-k}|\partial_{t}u+\Delta u|^{2}dxdt,\label{weight-estimate}
		\end{align}
		where $h(t)=te^{-t}$.
	\end{theorem}
	\begin{remark}
		The above inequality is a type of (time-space) weighted $L^{2}$ estimate. The corresponding  weight in Theorem \ref{weight} satisfies the following facts:
		
		\quad (i)\ The temporal weight $h(t)$ inspired from \cite{2004Backward} is sufficiently large when time is small enough.
		
		\quad (ii)\ For $k>\frac{3}{2}$, the spatial weight is integrable on the whole space.\\
		These two facts are crucial in our proof. 
	\end{remark}
	To prove Theorem \ref{weight}, we need the following lemma.
	\begin{lemma}\label{commutator}
		Assume $A$ is a densely defined operator on a real Hilbert space. Then for each $x\in\mathcal{D}(A^{2})$,the following holds,
		\begin{equation}
			\|Ax\|^{2}\geqslant \langle [J,K]x,x\rangle,
		\end{equation}
		where $J=\frac{A+A^{t}}{2}$ is the symmetric part of $A$, and $K=\frac{A-A^{t}}{2}$ is the skew symmetric part of $A$, $[J,K]=JK-KJ$.
	\end{lemma}
	\begin{proof}By direct computation,
		\begin{align*}
			\|Ax\|^{2}&=\|(J+K)x\|^{2}=\langle (J+K)x,(J+K)x\rangle\\
			&=\|Jx\|^{2}+\|Kx\|^{2}+\langle Jx,Kx\rangle+\langle Kx,Jx\rangle\\
			&\geqslant \langle Jx,Kx\rangle+\langle Kx,Jx\rangle\\
			&=\langle [J, K]x,x\rangle.
		\end{align*}
	\end{proof}
	Now we are ready to prove Theorem \ref{weight}.
	\begin{proof}[Proof of Theorem \ref{weight}]
		We consider 	
		\begin{align*}
			L:&=h(t)^{-(a+1)}(1+|x|^{2})^{-k}(\partial_{t}+\Delta)h^{a+1}(t)(1+|x|^{2})^{k}\\
			&=\partial_{t}+\Delta+(a+1)\frac{h'(t)}{h(t)}+\frac{4kx}{1+|x|^{2}}\cdot\nabla+\frac{2kn}{1+|x|^{2}}+\frac{4k(k-1)|x|^{2}}{(1+|x|^{2})^{2}}.
		\end{align*}
		Let $A=tL$. Then we can decompose $A$ as:
		\begin{align*}
			A=&t\partial_{t}+t\Delta+(a+1)(1-t)+\frac{4ktx}{1+|x|^{2}}\cdot\nabla+\frac{2knt}{1+|x|^{2}}+\frac{4k(k-1)t|x|^{2}}{(1+|x|^{2})^{2}}\\
			=&\bigg(t\Delta+\frac{4k^{2}t|x|^{2}}{(1+|x|^{2})^{2}}+(a+1)(1-t)-\frac{1}{2}\bigg)\\
			&+\bigg(t\partial_{t}+\frac{4ktx}{1+|x|^{2}}\cdot\nabla+2kt\Big(\frac{n}{1+|x|^{2}}-\frac{2|x|^{2}}{(1+|x|^{2})^{2}}\Big)+\frac{1}{2}\bigg)\\
			=&:J+K,
		\end{align*}
		where $J$ is symmetric and $K$ is skew-symmetric.
		Furthermore, the commutator of $J, K$ can be computed as follows:
		\begin{align*}
			[J,K]=&(a+1)t-t\Delta-\frac{4k^{2}t|x|^{2}}{(1+|x|^{2})^{2}}+\frac{8kt^{2}}{1+|x|^{2}}\Delta\\
			&-\frac{16kt^{2}}{(1+|x|^{2})^{2}}\Big[(x\otimes x)::\nabla^{2}\Big]+t^{2}V_{1}(x)x\cdot\nabla+t^{2}V_{2}(x),
		\end{align*}
		where
		\begin{align*}
			&V_{1}(x)=\frac{64k|x|^{2}}{(1+|x|^{2})^{3}}-\frac{16kn+32k}{(1+|x|^{2})^{2}},\\
			&V_{2}(x)=-\frac{(4n+8)kn}{(1+|x|^{2})^{2}}+\frac{(32kn+64k-32k^{3})|x|^{2}}{(1+|x|^{2})^{3}}+\frac{(64k^{3}-96k)|x|^{4}}{(1+|x|^{2})^{4}}.
		\end{align*}
		Let $v(x,t)=h(t)^{-(a+1)}(1+|x|^{2})^{-k}u(x,t)$, and applying lemma \ref{commutator}, we obtain that
		\begin{align*}
			\int_{\bR^{n}\times(T_{-},T^{+})}h(t)^{-2a}(1+|x|^{2})^{-2k}|\partial_{t} u+\Delta u|^{2}dxdt&\approx\|Av\|_{L^{2}}^{2}\geqslant \langle [J,K]v,v\rangle\\
			&=: \sum_{j=1}^{6} I_{j},
		\end{align*}
		where
		\begin{align*}
			&I_{1}=\int_{\bR^{n}\times(T_{-},T^{+})}(a+1)tv^2dxdt,\\
			&I_{2}=\int_{\bR^{n}\times(T_{-},T^{+})}-tv\Delta vdxdt,\\
			&I_{3}=\int_{\bR^{n}\times(T_{-},T^{+})}-\frac{4k^{2}t|x|^{2}}{(1+|x|^{2})^{2}}v^2dxdt,\\
		\end{align*}

		\begin{align*}
			&I_{4}=\int_{\bR^{n}\times(T_{-},T^{+})}\frac{8kt^{2}}{1+|x|^{2}}v\bigg(\Delta v-\Big[\frac{2(x\otimes x)}{1+|x|^{2}}::\nabla^{2}\Big]v\bigg)dxdt,\\
			&I_{5}=\int_{\bR^{n}\times(T_{-},T^{+})}t^{2}V_{1}(x)v(x\cdot\nabla) vdxdt,\\
			&I_{6}=\int_{\bR^{n}\times(T_{-},T^{+})}t^{2}V_{2}(x)v^2dxdt.
		\end{align*}
		It is easy to check that
		\begin{align*}
			&I_{1}\gtrsim (a+1)\int_{\bR^{n}\times(T_{-},T^{+})}h(t)^{-(2a+1)}(1+|x|^{2})^{-2k}|u|^{2}dxdt;\\
			&I_{3}\gtrsim -\int_{\bR^{n}\times(T_{-},T^{+})}h(t)^{-(2a+1)}(1+|x|^{2})^{-2k}|u|^{2}dxdt.
		\end{align*}
		For $I_{2}$, integrating by parts, we have
		\begin{align*}
			I_{2}&\gtrsim \int_{\bR^{n}\times(T_{-},T^{+})}h(t)^{-(2a+1)}\bigg|\nabla\Big((1+|x|^{2})^{-k}u\Big)\bigg|^2dxdt\\
			&\gtrsim \int_{\bR^{n}\times(T_{-},T^{+})}h(t)^{-(2a+1)}(1+|x|^{2})^{-2k}|\nabla u|^{2}dxdt\\
			&\quad-2\int_{\bR^{n}\times(T_{-},T^{+})}h(t)^{-(2a+1)}(1+|x|^{2})^{-2k}|u|^{2}dxdt.
		\end{align*}
		For $I_{4}$, integrating by parts again, using the fact that $T^{+}$ is small, we have
		\begin{align*}
			|I_{4}|&\lesssim \int_{\bR^{n}\times(T_{-},T^{+})}h(t)^{-2a}(1+|x|^{2})^{-1}\bigg|\nabla\Big((1+|x|^{2})^{-k}u\Big)\bigg|^2\\
			&\lesssim \int_{\bR^{n}\times(T_{-},T^{+})}h(t)^{-2a}(1+|x|^{2})^{-2k}|\nabla u|^{2}+\int_{\bR^{n}\times(T_{-},T^{+})}h(t)^{-2a}(1+|x|^{2})^{-2k}| u|^{2}dxdt,
		\end{align*}
		Similarly, integrating by parts, and applying Cauchy-Schwartz inequality, we can also obtain
		\begin{equation*}
			|I_{5}+I_{6}|\lesssim \int_{\bR^{n}\times(T_{-},T^{+})}h(t)^{-2a}(1+|x|^{2})^{-2k}|\nabla u|^{2}dxdt+\int_{\bR^{n}\times(T_{-},T^{+})}h(t)^{-2a}(1+|x|^{2})^{-2k}| u|^{2}dxdt.
		\end{equation*}
		Combining the above all estimates, using the fact that $a$ is large enough and $T^+$ is small enough, we have
		\begin{align*}
			&\int_{\bR^{n}\times(T_{-},T^{+})}h(t)^{-2a}(1+|x|^{2})^{-2k}|\partial_{t}u+\Delta u|^{2}dxdt\\
			&\quad\gtrsim
			(a+1)\int_{\bR^{n}\times(T_{-},T^{+})}h(t)^{-(2a+1)}(1+|x|^{2})^{-2k}|u|^{2}dxdt\\
			&\quad\quad+\int_{\bR^{n}\times(T_{-},T^{+})}h(t)^{-(2a+1)}(1+|x|^{2})^{-2k}|\nabla u|^{2}dxdt.
		\end{align*}
	\end{proof}
	By approximating with compactly support smooth functions, it holds that
	\begin{corollary}\label{general}
		Assume $u\in H_{0}^{1}\big((T_{-},T^{+});W^{2,p}(\R^{3})\big)$, where $2\le p\le \infty$, then \eqref{weight-estimate} still holds for $k>3\Big(\frac{1}{2}-\frac{1}{p}\Big)$.
	\end{corollary}
	\begin{proof}
		Let $\eta$ be a smooth spatial cut-off function such that $\eta\equiv 1$ for $|x|\le 1$, and $\eta\equiv 0$ for $|x|\ge 2$. Let $\eta_{m}(x)=\eta(\frac{x}{m})$. Replace $u$ by $\eta _{m}u$, and applying Theorem \ref{weight} we obtain
		\begin{align*}
			&\int_{\bR^{n}\times(T_{-},T^{+})}h(t)^{-(2a+1)}(1+|x|^{2})^{-k}\Big((a+1)u^{2}+|\nabla u|^{2}\Big)dxdt\\
			= &\lim_{m\to\infty}\int_{\bR^{n}\times(T_{-},T^{+})}h(t)^{-(2a+1)}(1+|x|^{2})^{-k}\Big((a+1)(\eta_{m}u)^{2}+|\nabla( \eta_{m}u)|^{2}\Big)dxdt\\
			\lesssim& \lim_{m\to\infty}\int_{\bR^{n}\times(T_{-},T^{+})}h(t)^{-2a}(1+|x|^{2})^{-k}|\partial_{t}\eta_{m}u+\Delta (\eta_{m}u)|^{2}dxdt.\\
			\lesssim &\lim_{m\to\infty}\int_{\bR^{n}\times(T_{-},T^{+})}h(t)^{-2a}(1+|x|^{2})^{-k}|\eta_{m}|^{2}|\partial_{t}u+\Delta u|^{2}dxdt\\
			&\quad +\lim_{m\to\infty}\int_{\bR^{n}\times(T_{-},T^{+})}h(t)^{-2a}(1+|x|^{2})^{-k}\big(|\Delta \eta_{m} u|^{2}+|\nabla \eta_{m}\nabla u|^{2}\big)dxdt\\
			\lesssim &\int_{\bR^{n}\times(T_{-},T^{+})}h(t)^{-2a}(1+|x|^{2})^{-k}|\partial_{t}u+\Delta u|^{2}dxdt.
		\end{align*}
	\end{proof}
	\section{Proof of the Main result}\label{pf}
	In this section, we will prove the main theorem of this paper. \begin{proof}[Proof of Theorem 1.1]
		Assume $u_1, u_2$ satisfy the assumption of  Theorem 1.1. Here, as in section \ref{carleman}, we choose $T_{-},T^{+}$ sufficiently small.  Let $u(x,t)=(u_1-u_2)(x,T+T_{-}-t)$, then $u$ satisfies the following  equation:
		\begin{equation}\label{back}
			\begin{cases}
				\partial_{t}u+\Delta u+\mathbb{P}\nabla\cdot(u_1\otimes u + u\otimes u_2)=0,\quad (x,t)\in\R^{3}\times [T_{-},T+T^{-}],\\
				\nabla\cdot u=0,\\
				u(\cdot, T_{-})=0.
			\end{cases}
		\end{equation}
		Then by Section 3 in \cite{2009SS}, $u$ is sufficiently smooth. Choose a temporal  cut-off function $\chi\in C^{\infty}[T_{-},T^{+}]$ such that $\chi(t)= 1$, when $t\in [T_{-},T_{-}+\delta]$, and $\chi(t)\equiv 0$ when $t\in [T_{-}+2\delta, T^{+}]$, where $3\delta<T^{+}-T_{-}$. For $\chi(t)u(x,t)$, applying  Corollary \ref{coro} with $\frac{3}{2}<k<\frac{5}{2}$, we obtain
		\begin{align*}
			&\int_{(T_{-},T^{+})\times \R^{3}}\chi^{2}(t)h(t)^{-(2a+1)}(1+|x|^{2})^{-k}\big((a+1)|u|^{2}+|\nabla u|^{2}\big)dxdt\\
			\lesssim &\int_{(T_{-},T^{+})\times \R^{3}}\chi^{2}(t)h(t)^{-2a}(1+|x|^{2})^{-k}|\partial_{t}u+\Delta u|^{2}dxdt\\
			&\quad +\int_{(T_{-},T^{+})\times \R^{3}}|\chi'(t)|^{2}h(t)^{-2a}(1+|x|^{2})^{-k}|u|^{2}dxdt\\
			\lesssim &\int_{(T_{-},T^{+})\times \R^{3}}\chi^{2}(t)h(t)^{-2a}(1+|x|^{2})^{-k}\Big|\mathbb{P}\nabla\cdot(u_1\otimes u + u\otimes u_2)\Big|^{2}dxdt\\
			&\quad +\int_{(T_{-},T^{+})\times \R^{3}}|\chi'(t)|^{2}h(t)^{-2a}(1+|x|^{2})^{-k}|u|^{2}dxdt\\
			:=&A+B.
		\end{align*}
		First, we estimate $A$. It turns out that $A$ can be absorbed in the left-hand side. Clearly,
		\begin{equation}\label{esti}
			\begin{aligned}
				A\lesssim &\int_{(T_{-},T^{+})\times \R^{3}}\chi^{2}(t)h(t)^{-2a}(1+|x|^{2})^{-k}|u_1\cdot \nabla u+u\cdot \nabla u_2|^{2}dxdt\\
				&+\int_{(T_{-},T^{+})\times\R^{3}}\chi^{2}(t)h(t)^{-2a}(1+|x|^{2})^{-k}|\mathcal{R}\nabla\cdot(u_1\otimes u )|^{2}dxdt\\
				&+\int_{(T_{-},T^{+})\times\R^{3}}\chi^{2}(t)h(t)^{-2a}(1+|x|^{2})^{-k}|\mathcal{R}\nabla\cdot( u\otimes u_2)|^{2}dxdt\\
				:=&A_{1}+A_{2}+A_{3},
			\end{aligned}
		\end{equation}
		where $\mathcal{R}=\nabla(-\Delta)^{-1}\nabla \cdot$ is the zeroth-order singular integral operator. Note that $A_{1}$ can be estimated directly,
		\begin{align*}
			A_{1}\lesssim &\int_{(T_{-},T^{+})\times\R^{3}}\chi^{2}(t)h(t)^{-2a}(1+|x|^{2})^{-k}|u_1|^{2}|\nabla u|^{2}dxdt\\
			&+\int_{(T_{-},T^{+})\times\R^{3}}\chi^{2}(t)h(t)^{-2a}(1+|x|^{2})^{-k}|u|^{2}|\nabla u_2|^{2}dxdt:=A_{11}+A_{12}.
		\end{align*}
		Clearly,
		\begin{equation*}
			A_{11}\lesssim \|u_1\|_{L^{\infty}_{t,x}}^{2}\int_{(T_{-},T^{+})\times\R^{3}}\chi^{2}(t)h(t)^{-2a}(1+|x|^{2})^{-k}|\nabla u|^{2}dxdt.
		\end{equation*}
		To ensure $\nabla u_2$ is not involved in the estimate of $A_{12}$, we have to rewrite the term $|\nabla u_2|^{2}$. Indeed,
		\begin{align*}
			A_{12}&=\int_{(T_{-},T^{+})\times\R^{3}}\chi^{2}(t)h(t)^{-2a}(1+|x|^{2})^{-k}|u|^{2}|\nabla u_2|^{2}dxdt\\
			&=\int_{(T_{-},T^{+})\times\R^{3}}\chi^{2}(t)h(t)^{-2a}(1+|x|^{2})^{-k}|u|^{2}\big(\Delta(|u_2|^{2})-u_2\cdot\Delta u_2\big)dxdt:=A_{13}+A_{14},
		\end{align*}
		Firstly, we estimate the term $A_{14}$. Note that $-\Delta u_2=\nabla\times\nabla \times u_2$ since $\nabla\cdot u_2=0$, integrating by parts, we have
		\begin{align*}
			A_{14}=&\int_{(T_{-},T^{+})\times\R^{3}}\chi^{2}(t)h(t)^{-2a}(1+|x|^{2})^{-k}|u|^{2}u_2\cdot(\nabla\times\nabla\times u_2)dxdt\\
			=&-\int_{(T_{-},T^{+})\times\R^{3}}\chi^{2}(t)h(t)^{-2a}\nabla\times\big((1+|x|^{2})^{-k}|u|^{2}u_2)\cdot\nabla\times u_2dxdt\\
			\lesssim &\int_{(T_{-},T^{+})\times\R^{3}}\chi^{2}(t)h(t)^{-2a}
			(1+|x|^{2})^{-(k+1/2)}|u|^{2}|u_2|\cdot|\nabla\times u_2|dxdt\\
			&+\int_{(T_{-},T^{+})\times\R^{3}}\chi^{2}(t)h(t)^{-2a}(1+|x|^{2})^{-k}|u|\cdot|\nabla u|\cdot|u_2|\cdot|\nabla\times u_2|dxdt\\
			&+\int_{(T_{-},T^{+})\times\R^{3}}\chi^{2}(t)h(t)^{-2a}(1+|x|^{2})^{-k}|u|^{2}|\nabla\times u_2|^{2}dxdt\\
			\lesssim &(\|u_2\|_{L^{\infty}_{t,x}}^{2}+\|\nabla\times u_2\|_{L^{\infty}_{t,x}}^{2})\int_{(T_{-},T^{+})\times\R^{3}}\chi^{2}(t)h(t)^{-2a}(1+|x|^{2})^{-k}(|u|^{2}+|\nabla u|^{2})dxdt.
		\end{align*}
		For $A_{13}$, integrating by parts again we have
		\begin{align*}
			A_{13}=&\int_{(T_{-},T^{+})\times\R^{3}}\chi^{2}(t)h(t)^{-2a}(1+|x|^{2})^{-k}|u|^{2}\big(\Delta(|u_2|^{2})\big)dxdt\\
			\le &\int_{(T_{-},T^{+})\times\R^{3}}\chi^{2}(t)h(t)^{-2a}\Big|\nabla\big((1+|x|^{2})^{-k}|u|^{2}\big)\Big|\cdot\Big|\nabla\big(|u_2|^{2}\big)\Big|dxdt\\
			\lesssim &\int_{(T_{-},T^{+})\times\R^{3}}\chi^{2}(t)h(t)^{-2a}(1+|x|^{2})^{-(k+1)}|u|^{2}|u_2|\cdot|\nabla u_2|dxdt\\
			&+\int_{(T_{-},T^{+})\times\R^{3}}\chi^{2}(t)h(t)^{-2a}(1+|x|^{2})^{-k}|u|\cdot|\nabla u|\cdot|u_2|\cdot |\nabla u_2|dxdt\\
			\lesssim &\epsilon \int_{(T_{-},T^{+})\times\R^{3}}\chi^{2}(t)h(t)^{-2a}(1+|x|^{2})^{-k}|u|^{2}|\nabla u_2|^{2}dxdt\\
			&+C_{\epsilon}\int_{(T_{-},T^{+})\times\R^{3}}\chi^{2}(t)h(t)^{-2a}(1+|x|^{2})^{-k}|u_2|^2(|u|^{2}+|\nabla u|^{2})dxdt\\
			\lesssim &\epsilon A_{12}+C_{\epsilon}  \|u_2\|_{L^{\infty}_{t,x}}^{2}\int_{(T_{-},T^{+})\times\R^{3}}\chi^{2}(t)h(t)^{-2a}(1+|x|^{2})^{-k}(|u|^{2}+|\nabla u|^{2})dxdt,
		\end{align*}
		where $\epsilon$ is sufficiently small and $C_{\epsilon}$ is a constant only depends on $\epsilon$.	Combining above computations, we have
		\begin{align*}
			A_{12}&\lesssim (\|u_2\|_{L^{\infty}_{t,x}}^{2}+\|\nabla\times u_2\|_{L^{\infty}_{t,x}}^{2})\int_{(T_{-},T^{+})\times\R^{3}} \chi^{2}(t)h(t)^{-2a}(1+|x|^{2})^{-k}(|u|^{2}+|\nabla u|^{2}).
		\end{align*}
		To estimate $A_{2}$ and $A_3$, we prove the following lemma:
		\begin{lemma}\label{lem-weight}
			for $0\le k < \frac{5}{2}$, the following inequality holds:
			\begin{equation}
				\int_{\R^{3}}(1+|x|^{2})^{-k}|\mathcal{R}\nabla \cdot f|^{2}dx \lesssim \int_{\R^{3}}(1+|x|^{2})^{-k}\Big(|\nabla f|^{2}+|f|^{2}\Big)dx .
			\end{equation}
		\end{lemma}
		\begin{remark}
			Note that for $k\ge \frac{3}{2}$, $(1+|x|^{2})^{-k}$ does not belong to any $\mathcal{A}_p$ class, however, the divergence structure enable us to improve the inequality to $k<\frac{5}{2}$, which enables $f\in L^\infty.$
		\end{remark}
		By the definition of $\mathcal{R}$ and divergence, for $1\le i,j,l \le 3,$ we define $$\mathcal{R}_{ijl}=\mathcal{F}^{-1}\frac{\xi_i \xi_j}{|\xi|^2}\mathcal{F}\partial_{x_l}$$
		via Fourier transform. Then with a slightly abuse of notations, we shall only estimate	
		\begin{equation*}
			I_{ijl}=\int_{\R^{3}}(1+|x|^{2})^{-k}|\mathcal{R}_{ijl}f|^{2}dx.
		\end{equation*}
		By the representation formula of homogeneous singular integral operator, for each $1\le i,j\le 3$, there exists a smooth function $\Omega_{ij}$ on $\mathbb{S}^2$ with mean zero such that $$\mathcal{R}_{ijl}f=\frac{\om(y/|y|)}{|y|^{3}} * \partial_{x_l}f,$$ where the convolution is defined in a limit sense. 
		Choosing $\varphi\in C_{c}^{\infty}(B_{2})$ and $\varphi\equiv 1$ in $B_{1}$, and integral by parts, we can decompose $\mathcal{R}_{ijl}$ into
		\begin{align*}
			\mathcal{R}_{ijl}f&=\bigg[\frac{\om(y/|y|)}{|y|^{3}}\varphi(y)\bigg] * \partial_{x_k}f - \partial_{y_l}\bigg[\frac{\om(y/|y|)}{|y|^{3}}(1-\varphi(y))\bigg] * f\\
			&:=\mathcal{R}_{ijl1}f+\mathcal{R}_{ijl2}f.
		\end{align*}
		The virtue of this decomposition is that $\mathcal{R}_{ijl1}$ is a localized operator while the kernal of $\mathcal{R}_{ijl2}$ is integrable. Hence we compute
		\begin{align*}
			I_{ijl}=&\int_{\R^{3}}(1+|x|^{2})^{-k}\big|\mathcal{R}_{ijl1}f+\mathcal{R}_{ijl2}f\big|^{2}dx\\
			\lesssim&\int_{\R^{3}}(1+|x|^{2})^{-k}\bigg|\mathcal{R}_{ijl1}f\bigg|^{2}dx\\
			&+\int_{\R^{3}}(1+|x|^{2})^{-k}\bigg|\mathcal{R}_{ijl2}f\bigg|^{2}dx\\
			:=&I_{ijl1}+I_{ijl2}.
		\end{align*}
		We will estimate $I_{ijl1},I_{ijl2}$ separately.\\
		\textit{Estimates of $I_{ijl1}$.} First, note that
		\begin{align*}
			(1+|x|^{2})^{-k/2}&=(1+|x-y|^{2})^{-k/2}-\int_{0}^{1}\frac{d}{dt}(1+|x-ty|^{2})^{-k/2}dt\\
			&=(1+|x-y|^{2})^{-k/2}-\int_{0}^{1}(1+|x-ty|^{2})^{-(k+2)/2}(x-ty)\cdot y dt.
		\end{align*}
		We decompose
		\begin{align*}
			I_{ijl1}\lesssim&\int_{\R^{3}}\bigg|\int_{\R^{3}}(1+|x-y|^{2})^{-k/2}\frac{\om(y/|y|)}{|y|^{3}}\varphi(y)\nabla f(x-y)dy\bigg|^{2}dx\\
			&+\int_{\R^{3}}\bigg|\int_{\R^{3}}\int_{0}^{1}(1+|x-ty|^{2})^{-(k+1)/2}dt\frac{\om(y/|y|)}{|y|^{2}}\varphi(y)\nabla f(x-y)dy\bigg|^{2}dx\\
			:=&I_{11}+I_{12}.
		\end{align*}
		By the $L^{2}$ boundness of singular integral integral,
		\begin{equation*}
			I_{11}\lesssim \int_{\R^{3}} (1+|x|^{2})^{-k}|\nabla f|^{2}dx.
		\end{equation*}
		For $I_{12}$, since for $|y|\le 2,$ we have
		\begin{equation*}
			1+|x-y|\lesssim 1+|x-ty|,
		\end{equation*}
		by Minkowski's inequality, we have
		\begin{align*}
			I_{12}&\lesssim \int_{\R^{3}}\int_{\R^{3}}\bigg|(1+|x-y|^{2})^{-(k+1)/2}\partial_{x_k}f(x-y)\bigg|^{2}\frac{\varphi(y)}{|y|^{2}}dydx\\
			&\lesssim \int_{\R^{3}}(1+|x|^{2})^{-k}|\nabla f|^{2}dx.
		\end{align*}
		\textit{Estimates of $I_{ijl2}$.}
		Note that $$\partial_{y_k}\bigg[\frac{\om(y/|y|)}{|y|^{3}}(1-\varphi(y))\bigg]\lesssim (1+|y|^{2})^{-2},$$
		we have
		\begin{align*}
			I_{ijl2}\lesssim&\int_{\R^{3}}(1+|x|^{2})^{-k}\bigg|\int_{\R^{3}}f(x-y)(1+|y|^{2})^{-2}dy\bigg|^{2}dx\\
			=&\int_{\R^{3}}\bigg|\int_{\R^{3}}f(x-y)(1+|x|^{2})^{-\frac{k}{2}}(1+|y|^{2})^{-2}dy\bigg|^{2}dx\\
			\lesssim&\int_{\R^{3}}\bigg|\int_{\R^{3}}f(x-y)(1+|x-y|^{2})^{-\frac{k}{2}}\bigg((1+|y|^{2})^{-2}+(1+|y|^{2})^{-(2-\frac{k}{2})}(1+|x|^{2})^{-\frac{k}{2}}\bigg)dy\bigg|^{2}dx\\
			\lesssim&\int_{\R^{3}}\bigg|\int_{\R^{3}}f(x-y)(1+|x-y|^{2})^{-\frac{k}{2}}(1+|y|^{2})^{-2}dy\bigg|^{2}dx\\
			&+\int_{\R^{3}}\bigg|\int_{\R^{3}}f(x-y)(1+|x-y|^{2})^{-\frac{k}{2}}(1+|y|^{2})^{-(2-\frac{k}{2})}(1+|x|^{2})^{-\frac{k}{2}}dy\bigg|^{2}dx\\
			:=&I_{21}+I_{22},
		\end{align*}
		where in the second inequality we  used the fact that $$(1+|x|^{2})^{-\frac{k}{2}}(1+|y|^{2})^{-2}\lesssim (1+|x-y|^{2})^{-\frac{k}{2}}\bigg((1+|y|^{2})^{-2}+(1+|y|^{2})^{-(2-\frac{k}{2})}(1+|x|^{2})^{-\frac{k}{2}}\bigg).$$
		For $I_{21}$, using Young's inequality we have 
		\begin{equation*}
			I_{21} \lesssim 	\int_{\R^{3}} (1+|x|^{2})^{-k}f^{2}(x)dx.
		\end{equation*}
		For $I_{22}$, since $k < \frac{5}{2}$, using Cauchy-Schwartz inequality  we have
		\begin{align*}
			I_{22}&= \int_{\R^{3}}(1+|x|^{2})^{-k}\bigg(\int_{\R^{3}}f(x-y)(1+|x-y|^{2})^{-\frac{k}{2}}(1+|y|^{2})^{-(2-\frac{k}{2})}dy\bigg)^{2}dx\\
			&\lesssim \left(\int_{\R^{3}}(1+|x|^{2})^{-\frac{5}{2}}dx\right)^{\frac{2k}{5}} \bigg\|\int_{\R^{3}}f(x-y)(1+|x-y|^{2})^{-\frac{k}{2}}(1+|y|^{2})^{-(2-\frac{k}{2})}dy\bigg\|_{L_x^{\frac{10}{5-2k}}(\R^{3})}^2.\\
		\end{align*}
		Then using Young's inequality, since $k<\frac{5}{2}$ implies $\frac{5(4-k)}{5-k}>3$, we have
		\begin{align*}
			I_{22}&\lesssim \bigg\|\int_{\R^{3}}f(x-y)(1+|x-y|^{2})^{-\frac{k}{2}}(1+|y|^{2})^{-(2-\frac{k}{2})}dy\bigg\|_{L_x^{\frac{10}{5-2k}}(\R^{3})}^2\\ &\lesssim \int_{\R^{3}} (1+|x|^{2})^{-k}f^{2}(x)dx \bigg\|(1+|y|^{2})^{-(2-\frac{k}{2})}\bigg\|_{L_x^{\frac{5}{5-k}}(\R^{3})}^2\\
			&\lesssim \int_{\R^{3}} (1+|x|^{2})^{-k}f^{2}(x)dx.
		\end{align*}
		Combining above estimates the lemma \ref{lem-weight} is proved. 
		
		The estimates of $A_2$ and $A_3$ now follows from $f=u_1\otimes u$ and $f=u\otimes u_2$. We have
		\begin{equation*}
			A_{2}\lesssim \int_{(T_{-},T^{+})\times \R^{3}}\chi^{2}(t)h(t)^{-2a}(1+|x|^{2})^{-k}(|u_1|^{2}|u|^{2}+|u_1|^{2}|\nabla u|^{2}+|\nabla u_1|^{2} |u|^{2})dxdt.
		\end{equation*}
		Similarly,
		\begin{equation*}
			A_{3}\lesssim \int_{(T_{-},T^{+})\times \R^{3}}\chi^{2}(t)h(t)^{-2a}(1+|x|^{2})^{-k}(|u_2|^{2}|u|^{2}+|u_2|^{2}|\nabla u|^{2}+|\nabla u_2|^{2}|u|^{2})dxdt.
		\end{equation*}
		Using the same techniques as in estimating $A_{12}$, we also have 
		\begin{align*}
			&A_2+A_3\\
			&\lesssim \big(\|u_1\|_{L^{\infty}_{t,x}}^{2}+\|\nabla\times u_1\|_{L^{\infty}_{t,x}}^{2}+\|u_2\|_{L^{\infty}_{t,x}}^{2}+\|\nabla\times u_2\|_{L^{\infty}_{t,x}}^{2}\big)\int_{(T_{-},T^{+})\times\R^{3}}\chi^{2}(t)h(t)^{-2a}(|u|^{2}+|\nabla u|^{2})dxdt.
		\end{align*}
		All in all, we have
		\begin{align*}
			&A=A_{1}+A_{2}+A_{3}\\
			&\lesssim \big(\|u_1\|_{L^{\infty}_{t,x}}^{2}+\|\nabla\times u_1\|_{L^{\infty}_{t,x}}^{2}+\|u_2\|_{L^{\infty}_{t,x}}^{2}+\|\nabla\times u_2\|_{L^{\infty}_{t,x}}^{2}\big)\int_{(T_{-},T^{+})\times\R^{3}}\chi^{2}h^{-2a}(|u|^{2}+|\nabla u|^{2})dxdt\\
			&\lesssim T^+\big(\|u_1\|_{L^{\infty}_{t,x}}^{2}+\|\nabla\times u_1\|_{L^{\infty}_{t,x}}^{2}+\|u_2\|_{L^{\infty}_{t,x}}^{2}+\|\nabla\times u_2\|_{L^{\infty}_{t,x}}^{2}\big)\int_{(T_{-},T^{+})\times\R^{3}}\chi^{2}h^{-(2a+1)}(|u|^{2}+|\nabla u|^{2})dxdt.
		\end{align*}
		Since $T^{+}$ can be chosen sufficenlty small, we have $$A\ll \int_{(T_{-},T^{+})\times \R^{3}}\chi^{2}(t)h(t)^{-(2a+1)}(1+|x|^{2})^{-k}\big((a+1)|u|^{2}+|\nabla u|^{2}\big)dxdt.$$
		As a result, we have
		\begin{align*}
			\int_{(T_{-},T^{+})\times \R^{3}}\chi^{2}(t)h(t)^{-(2a+1)}(1+|x|^{2})^{-k}\big((a+1)|u|^{2}+|\nabla u|^{2}\big)dxdt\\
			\lesssim \int_{(T_{-},T^{+})\times \R^{3}}|\chi'(t)|^{2}h(t)^{-2a}(1+|x|^{2})^{-k}|u|^{2}dxdt.
		\end{align*}
		By the definition of $\chi$,
		\begin{align*}
			&h(T_{-}+\delta/2)^{-(2a+1)}\int_{(T_{-},T_{-}+\frac{\delta}{2})\times\R^{3}}(1+|x|^{2})^{-k}|u|^{2}dxdt\\
			\leqslant &\int_{(T_{-},T^{+})\times \R^{3}}\chi^{2}(t)h(t)^{-(2a+1)}(1+|x|^{2})^{-k}|u|^{2}dxdt\\
			\lesssim &\int_{(T_{-},T^{+})\times \R^{3}}|\chi'(t)|^{2}h(t)^{-(2a+1)}(1+|x|^{2})^{-k}|u|^{2}dxdt\\
			\lesssim &h(T_{-}+\delta)^{-(2a+1)} \int_{(T_{-}+\delta,T_{-}+2\delta)\times \R^{3}}(1+|x|^{2})^{-k}|u|^{2}dxdt.
		\end{align*}
		Letting $a\to\infty$, we obtain $u\equiv 0$ on $[T_{-},T_{-}+\delta]$, or equivalently, $u_1\equiv u_2$  on $[T-\delta, T]$. Clearly, $\delta$ is independent of the terminal time $T$. Repeating the procedure for finite times, we obtain $u_1$ must coincide with $u_2$ at any time $t<T$.
	\end{proof}
	
	\section*{Acknowledgement}
	This work has been supported by the New Cornerstone Science Foundation through the XPLORER PRIZE, Sino-German Center Mobility Programme (Project No. M-0548), and Shanghai Science and Technology Program (Project No. 21JC1400600 and No. 19JC1420101).

	\frenchspacing
	\bibliographystyle{plain}
	\bibliography{backward_uniqueness_of_Navier-Stokes}

\end{document}